\documentclass[runningheads]{llncs}

\usepackage{amsmath, amssymb, graphicx, microtype, hyperref, tikz-cd}

\usepackage{comment, varwidth}
\usepackage{geometry}
\usepackage{subcaption}

\tikzcdset{scale cd/.style={every label/.append style={scale=#1},
    cells={nodes={scale=#1}}}}

\newcommand{\R}{\mathbb{R}}

\newcommand{\rank}{\mathrm{rank}}

\newtheorem{defi}{Definition}
\newtheorem{thm}{Theorem}
\newtheorem{lem}{Lemma}
\newtheorem{prop}{Proposition}

\newtheorem{cor}{Corollary}

\begin{document}

\title{On the Combinatorial Diameters\\ of Parallel and Series Connections}

\date{}

\author{Steffen Borgwardt\inst{1} \and Weston Grewe\inst{1} \and Jon Lee\inst{2}}


\institute{University of Colorado Denver \and University of Michigan, Ann Arbor}

\maketitle

\begin{abstract}
The investigation of combinatorial diameters of polyhedra is a classical topic in linear programming due to its connection with the possibility of an efficient pivot rule for the simplex method. We are interested in the diameters of polyhedra formed from the so-called parallel or series connection of oriented matroids: oriented matroids are the natural way to connect representable matroid theory with the combinatorics of linear programming, and these connections are fundamental operations for the construction of more complicated matroids from elementary matroid blocks.

\hspace{15pt} We prove that, for polyhedra whose combinatorial diameter satisfies the Hirsch-conjecture bound regardless of the right-hand sides in a standard-form description, the diameters of their parallel or series connections remain small in the Hirsch-conjecture bound. These results are a substantial step toward devising a diameter bound for all polyhedra defined through totally-unimodular matrices based on Seymour's famous decomposition theorem.

\hspace{15pt} Our proof techniques and results exhibit a number of interesting features. While the parallel connection leads to a bound that adds just a constant, for the series connection one has to linearly take into account the maximal value in a specific coordinate of any vertex. Our proofs also require a careful treatment of non-revisiting edge walks in degenerate polyhedra, as well as the construction of edge walks that may take a `detour' to facets that satisfy the non-revisiting conjecture when the underlying polyhedron may not. 
\end{abstract}
{\bf{Keywords}:} {combinatorial diameter, Hirsch conjecture, oriented matroid, parallel connection, series connection, 2-sum}
\\\\\noindent
{\bf{MSC}:} {52B05, 52B40, 52C40, 90C05}

\section{Introduction}

The \textit{combinatorial diameter}, or just \textit{diameter}, of a polyhedron is the maximum number of edges needed to form a walk between any two vertices of a polyhedron. Diameters of polyhedra is a classical field of study, almost as old as the inception of the simplex method, due to the possibility of providing a lower bound on the worst-case number of iterations of the simplex method. In particular, if there exists a polyhedron of exponential diameter (in the number of facets $f$ or dimension $d$), then there cannot exist an efficient pivot rule.

The famous Hirsch conjecture (see, e.g., \cite{d-63}) claimed that the diameter of a polyhedron with $f$ facets and dimension $d$ is bounded above by $f-d$. While disproved in general, known counterexamples \cite{kw-67,s-11} only violate it linearly. The `polynomial Hirsch conjecture', replacing the bound $f-d$ by a polynomial in $f-d$, remains open in general. Further, many important classes of polyhedra in combinatorial optimization, such as $0/1$ polytopes \cite{n-89} and network-flow polytopes \cite{bdf-17}, do satisfy the Hirsch-conjecture bound. For many other classes, validity of the Hirsch-conjecture bound remains open. The best known general bounds are quasi-polynomial, $f^{\log d +2}$ \cite{kk-92}, $(f-d)^{\log d}$ \cite{t-14}, $(f-d)^{\log O(\frac{d}{\log d})}$ \cite{s-19}, or polynomial in $d$ and the largest absolute value $\Delta$ of a subdeterminant of the constraint matrix, $O(\Delta^2 d^4 (\log d\Delta))$ (and $O(\Delta^2 d^{3.5}(\log d\Delta))$ for bounded polytopes)  \cite{bseh-12}. For the important class of polyhedra defined through totally-unimodular matrices, there is a well-known bound of $O(f^{16}d^3(\log fd)^3)$ \cite{df-94}, and, as $\Delta = 1$, the bounds of \cite{bseh-12} improve to $O(d^4 (\log d))$ and $O(d^{3.5}(\log d))$, respectively. For a survey, see \cite{ks-10}.

In the literature, the diameters of various polyhedra related to submodular functions
have been studied.
These include matroid polyhedra and polymatroids.
While the polyhedra that we study are generally  neither  matroid polyhedra nor polymatroids,
it is interesting to know that such
polyhedra do satisfy the Hirsch conjecture and more.
Matroid polyhedra are special $0/1$ polytopes, and thus the bound of Naddef applies \cite{n-89}. However, both types of polyhedra satisfy even stronger bounds; in particular the diameter has an upper bound of two times the size of the ground set \cite{t-92}.

The polyhedra that we study in what follows do not belong to these classes, but our contributions do strongly relate to matroids. 
We are interested in the diameters of polyhedra arising from the so-called parallel or series connection of matrices. These are the representable-matroid cases
of classical matroid operations to `connect' a pair of matroids on a single element from the ground set of each, to form a larger matroid (see \cite[Section 7.1]{oxley-06}). Both can be used to define the 2-sum of a pair of matroids. Because the matrices that we consider in linear programming are real (see \cite[Chapter 10]{bjo-99}), it is natural to also regard these operations as acting on oriented matroids (see \cite{l-89,l-90}, for example).


\subsection{Polyhedra of Parallel and Series Connections}

We begin by recalling the definitions of  parallel and series connections of matrices, and use it to define the equivalent notions for related polyhedra. Throughout, we consider two matrices $\bar{A}, \bar{B}$ with the following special forms: 
\[
\bar{A} := 
\begin{bmatrix}
    A & 0 \\
    a & 1
\end{bmatrix} \in \mathbb{R}^{m_1 \times n_1} \text { and } \bar{B} := 
\begin{bmatrix}
    1 & b \\
    0 & B
\end{bmatrix} \in \mathbb{R}^{m_2 \times n_2},
\]
where $A$ is a matrix in $\mathbb{R}^{(m_1-1) \times (n_1-1)}$, $a$ is a row vector in $\mathbb{R}^{n_1-1}$, $B$ is a matrix in $\mathbb{R}^{(m_2-1) \times (n_2-1)}$ and $b$ is a row vector in $\mathbb{R}^{n_2-1}$. The final column of $\bar{A}$ and first column of $\bar{B}$ are unit columns in $\mathbb{R}^{m_1}$ or $\mathbb{R}^{m_2}$  with a single entry $1$ in the final or first row, respectively.

First, we define the parallel and series connection for matrices $\bar{A}$ and $\bar{B}$. 

\begin{defi}[Parallel Connection]\label{def:P}
    The parallel connection of $\bar{A}$ and $\bar{B}$ is the matrix 
    \[
        P(\bar{A}, \bar{B}) := 
        \begin{bmatrix}
            A & 0 & 0 \\
            a & 1 & b \\
            0 & 0 & B
        \end{bmatrix}.
    \]
\end{defi}


\begin{defi}[Series Connection]\label{def:S}
    The series connection of $\bar{A}$ and $\bar{B}$ is the matrix 
    \[
        S(\bar{A}, \bar{B}) := 
        \begin{bmatrix}
            A & 0 & 0 \\
            a & 1 & 0 \\ 
            0 & 1 & b \\
            0 & 0 & B
        \end{bmatrix}.
    \]
\end{defi}

These connections are classical matroid operations (see \cite{oxley-06}) applied to represented matroids. From each matroid, one chooses an element from the ground set for the construction. These elements are represented in the matrices as the final column of $\bar{A}$ and the first column of $\bar{B}$, respectively. The specific form of $\bar{A}$ or $\bar{B}$ is not a restriction; any element of the ground set that is not a loop in the matroid can be chosen for the operation: through some elementary row operations and a reordering of the columns, any non-zero column of a matrix could be transformed into a unit column and moved to the first or final index, respectively, to obtain the form of $\bar{A}$ or $\bar{B}$. Once in this form, the operations can be applied.

If the matroids both happen to be graphic, then it is easy to describe the parallel connection and series connection
as operations on the underlying graphs. For this, we assume that we have a pair of graphs $G_1$ and $G_2$
with a unique edge $e$ that lies in both graphs.
We assign $e$ an orientation in each graph.
Parallel connection means gluing together $G_1$ and $G_2$ on $e$, respecting the orientation of $e$. Series connection is a bit more complicated: one glues $G_1\setminus e$ and $G_2\setminus e$ together on the endpoints of $e$, respecting the orientation of $e$, then splits the vertex at the head of $e$ into two vertices, and then joins these two vertices with an edge labeled $e$. An example of these operations is displayed in Figure \ref{par_series_fig}.

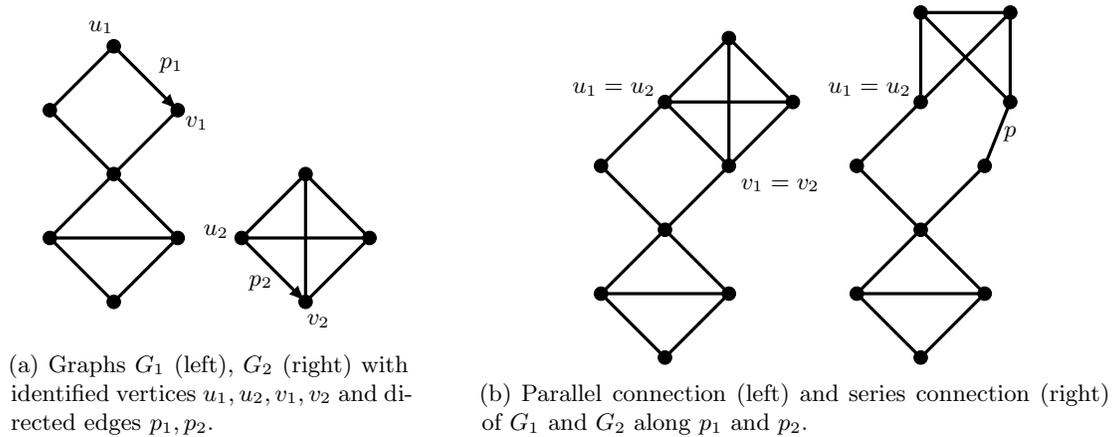
\begin{figure}
\centering
\begin{subfigure}[b]{.35\linewidth}
    \centering
    \begin{tikzpicture}[scale = 0.85]
        \draw[very thick] (1, 0) -- (0, 1) -- (1, 2) -- (0, 3) -- (1, 4);
        \draw[very thick, -latex] (1, 4) -- (2, 3);
        \draw[very thick] (2, 3) -- (1, 2) -- (2, 1) -- (1, 0);
        \draw[very thick] (0, 1) -- (2, 1);
    
        \draw[fill = black] (1, 0) circle (3pt);
        \draw[fill = black] (0, 1) circle (3pt);
        \draw[fill = black] (2, 1) circle (3pt);
        \draw[fill = black] (1, 2) circle (3pt);
        \draw[fill = black] (0, 3) circle (3pt);
        \draw[fill = black] (2, 3) circle (3pt);
        \draw[fill = black] (1, 4) circle (3pt);
    
        \node at (1.9, 3.7) {\small $p_1$};
        \node at (0.8, 4.3) {\small $u_1$};
        \node at (2.3, 2.8) {\small $v_1$};

        \draw[very thick] (3, 1) -- (4, 2) -- (5, 1) -- (4, 0);
        \draw[very thick, -latex] (3, 1) -- (4, 0);
        \draw[very thick] (4, 0) -- (4, 2);
        \draw[very thick] (3, 1) -- (5, 1);
        
        \draw[fill = black] (4, 0) circle (3pt);
        \draw[fill = black] (3, 1) circle (3pt);
        \draw[fill = black] (5, 1) circle (3pt);
        \draw[fill = black] (4, 2) circle (3pt);
    
        \node at (3.3, .3) {\small $p_2$};
        \node at (2.6, 1.1) {\small $u_2$};
        \node at (4.2, -0.3) {\small $v_2$};
    \end{tikzpicture}
    \caption{Graphs $G_1$ (left), $G_2$ (right) with identified vertices $u_1, u_2, v_1, v_2$ and directed edges $p_1, p_2$.}
\end{subfigure}
\hspace{.05\linewidth}
\begin{subfigure}[b]{.55\linewidth}
    \centering
    \begin{tikzpicture}[scale = 0.85]
        \draw[very thick] (1, 0) -- (0, 1) -- (1, 2) -- (0, 3) -- (1, 4);
        \draw[very thick] (1, 4) -- (2, 3);
        \draw[very thick] (2, 3) -- (1, 2) -- (2, 1) -- (1, 0);
        \draw[very thick] (0, 1) -- (2, 1);
        \draw[very thick] (1, 4) -- (2, 5) -- (3, 4) -- (2, 3);
        \draw[very thick] (2, 3) -- (2, 5);
        \draw[very thick] (1, 4) -- (3, 4);
    
        \draw[fill = black] (1, 0) circle (3pt);
        \draw[fill = black] (0, 1) circle (3pt);
        \draw[fill = black] (2, 1) circle (3pt);
        \draw[fill = black] (1, 2) circle (3pt);
        \draw[fill = black] (0, 3) circle (3pt);
        \draw[fill = black] (2, 3) circle (3pt);
        \draw[fill = black] (1, 4) circle (3pt);
        \draw[fill = black] (2, 5) circle (3pt);
        \draw[fill = black] (3, 4) circle (3pt);

        \node at (0.2, 4.2) {\small $u_1 = u_2$};
        \node at (2.8, 2.7) {\small $v_1 = v_2$};
    
        \draw[very thick] (5, 0) -- (4, 1) -- (5, 2) -- (4, 3) -- (5, 4);
        \draw[very thick] (6, 3) -- (5, 2) -- (6, 1) -- (5, 0);
        \draw[very thick] (4, 1) -- (6, 1);
        \draw[very thick] (5, 4) -- (5, 5.4) -- (6.4, 5.4) -- (6.4, 4) -- (6, 3);
        \draw[very thick] (5, 4) -- (6.4, 5.4);
        \draw[very thick] (5, 5.4) -- (6.4, 4);
    
        \draw[fill = black] (5, 0) circle (3pt);
        \draw[fill = black] (4, 1) circle (3pt);
        \draw[fill = black] (6, 1) circle (3pt);
        \draw[fill = black] (5, 2) circle (3pt);
        \draw[fill = black] (4, 3) circle (3pt);
        \draw[fill = black] (6, 3) circle (3pt);
        \draw[fill = black] (5, 4) circle (3pt);
        \draw[fill = black] (5, 5.4) circle (3pt);
        \draw[fill = black] (6.4, 5.4) circle (3pt);
        \draw[fill = black] (6.4, 4) circle (3pt);
    
        \node at (4.2, 4.2) {\small $u_1 = u_2$};
        \node at (6.4, 3.5) {\small $p$};
    \end{tikzpicture}
    \caption{Parallel connection (left) and series connection (right) of $G_1$ and $G_2$ along $p_1$ and $p_2$.}
\end{subfigure}
\caption{Example of Parallel and Series Connections of Graphs}\label{par_series_fig}
\end{figure}

There is a further important operation called the 2-sum, and for graphs it is equivalently realized by deleting $e$ in the parallel connection or by contracting $e$ in the series connection. For matrices $\bar{A}$ and $\bar{B}$, where we assume that $e$ labels the last column of $\bar{A}$ 
and the first column of  $\bar{B}$, 
we can realize the matrix for the 2-sum by deleting the
standard-unit vector column from $ P(\bar{A}, \bar{B})$. 
Parallel connection and series connection (and hence 2-sum) 
preserve total unimodularity. Considering matroids,
2-sum preserves the property of being `identically self dual' (see \cite[Exercise 4, Section 7.1]{oxley-06}). 
2-sum is also a key property for working with the important concept of connectivity in
matroids \cite[Theorem 8.3.1]{oxley-06}: a `2-connected' matroid $N$ is not `3-connected'
if and only if it is the 2-sum of proper minors of $M$. The Fano matroid is a fundamental matroid 
for studying matroids of totally-unimodular matrices. It and its dual characterize, by exclusion as minors, 
binary matroids that are not representable by totally-unimodular matrices. It is an important result
of Seymour that
every binary matroid that does not have the dual of the Fano matroid as a minor can be constructed
by direct sums and 2-sums, starting with matroids that are representable by totally-unimodular matrices and
copies of the Fano matroid (see \cite[Corollary 11.2.5]{oxley-06}). Additionally, Cunningham and Edmonds demonstrated how 2-sum is the
basic operation for constructing all 2-connected matroids, using a matroid decomposition tree, starting from uniform matroids each of which has rank or co-rank 1 (see \cite{ot-14}, for example).

As for linear programming, the operations parallel connection, series connection,  
and 2-sum arise naturally in models
involving time-staged decision making. For example, this happens for situations in which the activities in
 time period $t+1$ depend on the activities in
 time period $t$ through a single shared resource. This can be naturally modeled
 via 2-sum, while parallel connection allows for injecting more of that 
 resource into the system at the start of time period $t+1$, at a (time-dependent) per-unit cost
 that would be captured in the objective function. 
 Series connection can also be relevant 
 for linear programming: in a situation in which 
 there are separate endowments of a resource in time periods $t$ and $t+1$,
 one might want to match the levels of leftover resources in the two
 time periods (to perhaps balance risk). 
All of these situations are special cases of `staircase linear programs',
 which are very challenging for the classical simplex method. 
 In fact, the simplex method has been specially tailored
 to take advantage of the general staircase form (for example, see \cite{amt-85,fo-82,fo-83} and the references therein). 

We are interested in relating the diameters of polyhedra defined through a parallel and series connection of the main constraint matrix to the diameters of the original polyhedra. Let us introduce the formal notation for this discussion. 

Throughout, let 
\begin{equation*} \label{QRForm}
Q := 
\left\{ x \in \R^{n_1} : \bar{A}x = 
\begin{pmatrix}
    c_A \\
    c_a
\end{pmatrix}, 
x \geq 0 \right\}
\text{ and }
R := 
\left\{ x \in \R^{n_2} : \bar{B}x = 
\begin{pmatrix}
    c_b \\
    c_B
\end{pmatrix}, 
x \geq 0 \right\},
\end{equation*}
where $c_A$ and $c_B$ are column vectors in $\mathbb{R}^{(m_1-1)}$ and $\mathbb{R}^{(m_2-1)}$, respectively, and $c_a, c_b \in \mathbb{R}$. 

The polyhedra that we study arise from $Q$ and $R$ through a connection of their main constraint matrices and a corresponding merging of the right-hand side vectors.

\begin{defi}\label{def:PQR}
    The parallel-connection polyhedron of $Q$ and $R$ is the polyhedron 
    \[
        P(\bar{A}, c_{\bar{A}}; \bar{B}, c_{\bar{B}}) := 
        \left\{ 
            x \in \R^{n_1 + n_2 - 1} : P(\bar{A}, \bar{B})x = 
            \begin{pmatrix}
                c_A \\
                c_a + c_b \\
                c_B
            \end{pmatrix},
            x \geq 0
        \right\}.
    \]
\end{defi}

\begin{defi} \label{def:SQR}
    The series-connection polyhedron of $Q$ and $R$  is the polyhedron
    \[
        S(\bar{A}, c_{\bar{A}}; \bar{B}, c_{\bar{B}}) := 
        \left\{ 
            x \in \R^{n_1 + n_2 - 1} : S(\bar{A}, \bar{B})x = 
            \begin{pmatrix}
                c_A \\
                c_a \\
                c_b \\
                c_B
            \end{pmatrix},
            x \geq 0
        \right\}.
    \]
\end{defi}

For convenience, we let $\mathcal{P} := P(\bar{A}, c_{\bar{A}}; \bar{B}, c_{\bar{B}})$ and $\mathcal{S} := S(\bar{A}, c_{\bar{A}}; \bar{B}, c_{\bar{B}})$.
Our goal is to establish diameter bounds for $\mathcal{P}$ and $\mathcal{S}$ in terms of the diameters of the original polyhedra $Q$ and $R$. We will refer to the diameter of a polyhedron $Q$ as $\operatorname{diam}(Q)$. We seek to understand how the diameter of a polyhedron with constraint matrix $P(\bar{A}, \bar{B})$ or $S(\bar{A}, \bar{B})$ is related to the diameters of polyhedra with constraint matrix $\bar{A}$ or $\bar{B}$.
Thus, we are specifically interested in finding diameter bounds that do not depend on the right-hand side. Our notation for a diameter bound for a class of polyhedra with the same constraint matrix and varying right-hand sides follows:


\begin{defi}
Let $A$ be a matrix in $\mathbb{R}^{m \times n}$. Then
\[
    \operatorname{diam}(A) := \max_{b \in \R^m}
   \left \{\operatorname{diam} \{ x \in \R^n : Ax = b, x \geq 0 \} \right\}.
\]
\end{defi}
Informally, we overload  `$\operatorname{diam}()$' to be able to use it for a specific polyhedron or for a matrix, to refer to the class of polyhedra sharing it as their main constraint matrix.

There are some immediate benefits for using $\operatorname{diam}(A)$ as opposed to just the diameter for a polyhedron. First, we may assume that $P$ and $Q$ are simple, as it is known that the maximum diameter (among polyhedra with the same number of facets and dimension) is achieved by a simple polyhedron. Indeed, given any degenerate polyhedron, there exists a right-hand side perturbation that can only increase the diameter; see, e.g., \cite{ykk-84}. Further, when proving a diameter bound for a connection of $P$ and $Q$, we now have the freedom to discuss diameters for polyhedra with the same constraint matrix as $P$ or $Q$ but with a different right-hand side. 


Finally, we always assume standard form and irredundancy of the representation. For an $m \times n$ constraint matrix, this implies that the number of facets is $n$ and the dimension is $n-m$, as opposed to the usual $n-\rank(A)$ when the equality constraints are possibly redundant. Irredundancy is not a restriction for two reasons: first, the diameter of a polyhedron is not affected by its representation in general. Second, the final row of $\bar{A}$ and the first row of $\bar{B}$ are the only ones taking a `special role' in a parallel or series connection; however, note that for the parallel connection, these rows keep linear independence of any independent row subset from the other rows due to the isolated $1$-entry in the final or first column, respectively. For the series connection, it is possible that the rows of $S(\bar{A}, \bar{B})$ are not linearly independent, even if the rows of $\bar{A}$ and $\bar{B}$ are linearly independent. It is possible that both linking constraints can be row-reduced to $(0, 1, 0)$. However, this would imply that the polyhedron is a Cartesian product of two polyhedra with diameters $d(\bar{A})$ and $d(\bar{B})$. Thus, the diameter is bounded by $d(\bar{A}) + d(\bar{B})$, which will still be covered by our bounds. Our proof techniques do not immediately transfer to this case. Thus, we assume that $S(\bar{A}, \bar{B})$ has linearly-independent rows.

%


\subsection{Contributions and Outline}\label{sec:contributions}

Our main results are diameter bounds for all parallel- and series-connection polyhedra. In Section \ref{PCSec}, we discuss parallel connections; in Section \ref{SeriesSec}, we discuss series connections. 

These discussions crucially depend on an understanding of the underlying graphs of the polyhedra $\mathcal{P}$ and $\mathcal{S}$. For each construction, we begin by characterizing and categorizing the vertices in terms of `counts' of basic and nonbasic variables relating to the original polyhedra $Q,R$. Because walking along an edge to an adjacent vertex corresponds to the exchange of a single basic and nonbasic variable (for some basis of the involved vertices), we also obtain a better understanding of the edges of $\mathcal{P}$ and $\mathcal{S}$. The categorization of the vertices and insight into the underlying graph enables us to prove the diameter bounds for the parallel- and series-connection polyhedra. 

In Section \ref{secPCCLass}, we begin with discussing the generality of our results. In particular, we show the parallel connection arises as a special case of the matrix $2$-sum operation (see \cite{s-98}). Then, we categorize the vertices of the parallel-connection polyhedron. Writing a vertex of the parallel connection as the $(n_1+n_2-1)$-vector $(x, s, y)$ where $x$ is a $(n_1-1)$-vector, $s$ is a scalar, and $y$ is a $(n_2-1)$-vector. Any vertex must contain exactly $m_1 + m_2 - 1$ basic variables. Further, simplicity of $Q, R,$ and $\mathcal{P}$ implies $x$ and $y$ must have at least $m_1 - 1$ and $m_2-1$ (nonzero) basic variables, respectively. This gives a degree of freedom as to where to place the final basic variable.  

    


In Section \ref{PCDiaBd}, we exploit the categorization to show that the diameter of the parallel connection of two polyhedra $Q$ and $R$, with diameters bounded above by $\text{diam}(\bar{A})=d(\bar{A})$ and $\text{diam}(\bar{B})=d(\bar{B})$, respectively, is at most $d(\bar{A}) + d(\bar{B}) + 2$. This is stated in Theorem \ref{PCDia}. 

\begin{thm}[Diameter - Parallel Connection] \label{PCDiaRes}
The combinatorial diameter of $\mathcal{P}$ is at most $d(\bar{A}) + d(\bar{B}) + 2$.
\end{thm}

Our techniques involve identifying walks in $Q$ and $R$ and showing that copies of these walks exist in $\mathcal{P}$. We say a walk in $Q$, denoted $(x^1, s^1), (x^2, s^2), \ldots, (x^k, s^k)$, {\em lifts} to $\mathcal{P}$ if there exists $y$ such that $(x^1, s^1, y), (x^2, s^2, y), \ldots, (x^k, s^k, y) \in \mathcal{P}$ are vertices and $(x^i, s^i, y)$ is adjacent to $(x^{i+1}, s^{i+1}, y)$ for $1 \leq i < k.$ 

Throughout, we are interested in connecting our results to the Hirsch conjecture. Thus, in a corollary we assume that the diameters of $Q$ and $R$ are at most the Hirsch bound. When we make this assumption we say {\em $Q$ and $R$ satisfy Hirsch}, or simply {\em $Q$ and $R$ are Hirsch}. If $Q$, and $R$ satisfy Hirsch for all right-hand sides, then $\text{diam}(Q)\leq m_1$ and $\text{diam}(R)\leq m_2$, where $m_1, m_2$ are the number of rows in the main constraint matrices. The diameter of $\mathcal{P}$ is at most the claimed Hirsch bound ($m_1 + m_2 - 1$) plus 3, by Theorem \ref{PCDia}.



In Section \ref{SeriesSec}, we study the series-connection polyhedron. The discussion requires additional technical tools and, as we will explain, leads to two different types of bounds. In particular, to devise general diameter bounds, we require that the polyhedra $Q, R$ allow for so-called {\em non-revisiting} walks that adhere to the specified bounds. We begin by recalling the intimate connection of non-revisiting walks, the non-revisiting conjecture, and the Hirsch conjecture. In Section \ref{sec:classseries}, we present a categorization of the vertices of series-connection polyhedra. This categorization is similar to the categorization we used for the parallel-connection polyhedron. We also discuss the generality of our results; in particular, we demonstrate that the series connection of any pair of matrices may be reduced to the form $S(\bar{A}, \bar{B})$.


Again, the categorization of the vertices of $\mathcal{S}$ is one of the main tools we use to prove our stated bounds for the diameter of the series-connection polyhedron. For the first bound (Theorem \ref{thm:series2}), we assume that $Q, R$ are integral polyhedra, i.e., have integer coordinates at all vertices. This is not a restriction for rational input. 

\begin{thm}[Diameter - Series Connection]\label{thm:series2}
Let $Q$ and $R$ be integral polyhedra that satisfy the non-revisiting conjecture. Let $\mathcal{S}$ be the series-connection polyhedron for $Q$ and $R$ and $s_{\min} := \min \{s : (x, s, y) \in \mathcal{S}\}$.
If $(x^1, s^1, y^1), (x^2, s^2, y^2) \in \mathcal{S}$ are vertices, then the distance between $(x^1, s^1, y^1)$ and $(x^2, s^2, y^2)$ is at most $m_1 + m_2 + s^1 + s^2 - 2s_{\min}.$ If $\sup \{s : (x, s, y) \in \mathcal{S}\} = s_{\max} < \infty,$ then the diameter of $\mathcal{S}$ is bounded by $m_1 + m_2 + s_{\text{diff}}$, where $s_{\rm{diff}} = s_{\max} - s_{\min}$.
\end{thm}

Unlike for the parallel-connection polyhedron, where we are able to prove a constant, additive bound, the sum of the original bounds here may be exceeded by the term $s_{\rm{diff}}$, which represents the range of possible values for the shared variable $s$. This makes it a different type of bound that has to represent $s_{\rm{diff}}$ in the input to claim linearity. The excess is caused by walking the shared variable $s$ to an extreme value, which allows a transfer of non-revisiting walks from the original polyhedra $Q,R$ to (parts of) the walk in $\mathcal{S}$. If $Q$ and $R$ are Hirsch, we can see that the diameter exceeds Hirsch by at most $s_{\rm{diff}}.$


To provide a bound which does not depend on $s_{\rm{diff}}$, we devise a new approach to construct walks between vertices of $\mathcal{S}$ in Section \ref{sec:series_bound}. We introduce a new notion of diameter, the $s$-bounded diameter (Definition \ref{def:boundeddiam}, stylized $d_b(\cdot)$), that characterizes the maximum distance between two vertices when we do not allow the shared variable $s$ to drop below a specified value. The $s$-bounded diameter is a weaker assumption than the use of the monotone diameter \cite{todd2002many}, where (strict) improvement with respect to a given objective function is required. We are able to show that the diameter is at most quadratic in the diameters and the $s$-bounded diameters of the original polyhedra, $Q$ and $R$. The work culminates in the following theorem:

\begin{thm} 
The diameter of the series-connection polyhedron, $\mathcal{S},$ where $s_{\min}^R \leq s_{\min}^Q$ is
at most the maximum of 
\[
    \{d(\bar{A})(d_b(\bar{B})+1) + d(B) + 2d(\bar{B}), \,
    (d(\bar{A})-1)(d_b(\bar{B})+1) + d(\bar{A}) + 3d(\bar{B}) + 2 \}.
\]
\end{thm}

A similar bound holds when $s_{\min}^R \geq s_{\min}^Q$ by interchanging the roles of $Q$ and $R$. Again, we can view these bounds in terms of the Hirsch conjecture. In this case, the bound simplifies slightly because the diameter of a face cannot exceed the diameter of the polyhedron when Hirsch is satisfied for all right-hand sides, and thus we replace $d(B)$ with $d(\bar{B}).$


In Section \ref{sec:conclusion}, we provide a brief conclusion and give an outlook on some natural directions for future research. 

\setcounter{thm}{0}
\setcounter{cor}{0}


\section{Parallel Connection} \label{PCSec}

In this section, we prove that the diameter of the parallel-connection polyhedron of two polyhedra $Q,R,$ with diameters bounded by $d(\bar{A})$, $d(\bar{B})$, respectively, is bounded above by $d(\bar{A})+d(\bar{B})+2$. Consequently, we establish that the diameter of a parallel-connection polyhedron of two Hirsch-satisfying polyhedra does not exceed Hirsch by more than $3$. 

Throughout, we refer to points of $\mathcal{P}$ as triples $(x, s, y)$ where $x \in \R^{n_1 - 1}, s \in \R, y \in \R^{n_2 - 1}$. This allows us to write  
    \[
        \mathcal{P} = 
        \left\{ 
            \begin{pmatrix}
                x \\
                s \\
                y
            \end{pmatrix} \in \R^{n_1 + n_2 - 1} : 
        \begin{bmatrix}
            A & 0 & 0 \\
            a & 1 & b \\
            0 & 0 & B
        \end{bmatrix} \begin{pmatrix}
                x \\
                s \\
                y
            \end{pmatrix} =
            \begin{pmatrix}
                c_A \\
                c_a + c_b \\
                c_B
            \end{pmatrix},
            x, s, y \geq 0
        \right\}.
    \]
The scalar $s$ is associated with the `linking' column or variable. For this reason, we refer to $s$ as the \textit{shared variable.} Additionally, we recall that $\bar{A}$ and $\bar{B}$ have full rank and $Q$ and $R$ are simple. 

Our results for the parallel connection can be viewed as a special case of the $2$-sum. The $2$-sum \cite{s-98} of two general matrices is defined as:
\[
    \begin{bmatrix}
        A & a 
    \end{bmatrix}
    \oplus_2
    \begin{bmatrix}
        b \\
        B
    \end{bmatrix}
    =
    \begin{bmatrix}
        A & ab \\
        0 & B
    \end{bmatrix}.
\]
In fact, by the definition of parallel connection, we have $P(\bar{A}, \bar{B}) = \bar{A} \oplus_2 \bar{B}$. Furthermore, if $M$ is any matrix with $M_{m,n} \neq 0$, then pivoting on the shared column $M \oplus_2 \bar{B}$ yields a parallel connection of two matrices. Thus, if a polyhedron can be represented in standard form with main constraint matrix $M \oplus_2 \bar{B}$, then our results for parallel connection hold for this special case of the $2$-sum. However, given two general matrices $M_1$ and $M_2$, $M_1 \oplus_2 M_2$ may not be able to be reduced to the form $P(\bar{A}, \bar{B})$ for some $\bar{A}, \bar{B}$. So, our results do not generalize completely to the $2$-sum. 



\subsection{Categorization of Vertices} \label{secPCCLass}
Our first step is to establish an understanding of the graph of $\mathcal{P}$; see Definitions \ref{def:P} and \ref{def:PQR}. Toward this end, we categorize the vertices of $\mathcal{P}$. 

We begin with an observation on the number of basic variables in the blocks $x$ and $y$.
If $(x, s, y) \in \mathcal{P}$ is a vertex and $\mathcal{P}$ is simple, then the vertices of $\mathcal{P}$ are in one-to-one correspondence with the basic feasible solutions of $\mathcal{P}$. It follows that $x$ has at least $m_1 - 1$ (nonzero) basic variables and $y$ has at least $m_2 - 1$ (nonzero) basic variables, because $m_1-1$ and $m_2-1$ basic variables are needed to satisfy $Ax = c_A$ and $By = c_B$, respectively, by simplicity of $Q$ and $R$. Because $\mathcal{P}$ is simple, a basic feasible solution has exactly $m_1 + m_2 - 1$ (nonzero) basic variables, and so there is a single extra basic variable to `place'. This allows us to partition the vertices into $3$ categories.

\begin{prop}[Categorization of Vertices] \label{PCClass}
If $\mathcal{P}$ is simple, then the vertices of $\mathcal{P}$ can be partitioned into $3$ categories based on a split of basic variables of the corresponding basic feasible solutions. 
\begin{enumerate}
    \item The first category is $(m_1, 0, m_2 -1)$ indicating those vertices $(x, s, y)$ where $x$ has $m_1$ basic variables, $s$ is nonbasic, and $y$ has $m_2 -1$ basic variables. 
    
    \item The second category is $(m_1-1, 1, m_2 -1)$ indicating those vertices $(x, s, y)$ where $x$ has $m_1-1$ basic variables, $s$ is basic, and $y$ has $m_2 -1$ basic variables.

    \item The third category is $(m_1-1, 0, m_2)$ indicating those vertices $(x, s, y)$ where $x$ has $m_1-1$ basic variables, $s$ is nonbasic, and $y$ has $m_2$ basic variables.
\end{enumerate}
\end{prop}

We call the category a vertex belongs to its {\em basis split}.
The categorization into different basis splits will be a crucial tool for obtaining a diameter bound for the parallel-connection polyhedron. We are able to establish a diameter bound by finding a bound for the distance of vertices from each pair of categories.


\subsection{Diameter Bound} \label{PCDiaBd}

Next, we prove that the diameter of $\mathcal{P}$ is at most $d(\bar{A}) + d(\bar{B}) + 2$. Our strategy is to consider walks in the original polyhedra $Q, R$ (but for possibly different right-hand sides) and show that we can lift and concatenate those walks to walks in $\mathcal{P}$ between corresponding vertices. We first prove two technical lemmas and a corollary which enable us to prove the main result (Theorem \ref{PCDia}). We will make use of two sets of parameterized polyhedra, $Q(t)$ and $R(t)$:
\[
    Q(t) = \left\{ 
        x \in \R^{n_1} : \bar{A}x = \begin{pmatrix}
            c_A \\
            t
        \end{pmatrix}, x \geq 0
    \right\}, \quad
    R(t) = \left\{ 
        x \in \R^{n_2} : \bar{B}x = \begin{pmatrix}
            t \\
            c_B
        \end{pmatrix}, x \geq 0
    \right\}
\]
for all $t \in \R.$ The diameters of $Q(t)$ and $R(t)$ are bounded above by $\operatorname{diam}(\bar{A})$ and $\operatorname{diam}(\bar{B})$ because they have $\bar{A}$ and $\bar{B}$ as equality constraint matrices, respectively. We begin with a lemma that demonstrates how to lift a walk in $Q(t)$, for a certain $t$, to $\mathcal{P}$.


\begin{lem}\label{lift_lem}
Let $\mathcal{P}$ be simple. Let $(x, s, y) \in \mathcal{P}$ be a vertex and set $t = c_a + c_b - by.$ If $y$ has $m_2-1$ basic variables, then $Q(t) \times \{ y \}$ is a face of $\mathcal{P}$ with diameter at most $d(\bar{A}).$ Thus, if $(x^1, s^1, y), (x^2, s^2, y) \in \mathcal{P}$ are vertices then the distance between $(x^1, s^1, y)$ and $(x^2, s^2, y)$ is at most $d(\bar{A}).$
\end{lem}

\begin{proof}
Let $(x^*, s^*, y^*) \in \mathcal{P}$ be such a vertex. The set $Q(t) \times \{ y^* \}$ is a face because of the following two facts. First, $Q(t) \times \{ y^* \}$ is a subset of the polyhedron that arises from forcing a collection of constraints to be satisfied with equality (i.e. $y_i = 0$ if $y_i^* = 0$). Second, by simplicity, there is no other $y$ with support contained in the support $y^*$ such that $By = c_B$. The diameter of $Q(t) \times \{ y^* \}$ is $d(\bar{A})$ because it is the Cartesian product of $Q(t)$ and a point.
\qed
\end{proof}

Lemma \ref{lift_lem} allows us to lift walks from a polyhedron $Q(t)$ to the parallel-connection polyhedron $\mathcal{P}$. In turn, we may transfer any known bounds on the lengths of those walks. By switching the roles of $Q$ and $R$ (or $x$ and $y$), it follows that walks from a polyhedron $R(t)$, for a particular $t$, can be lifted to $\mathcal{P}$, as well.

To prove the stated diameter bound for the parallel-connection polyhedron, we split the discussion into cases based on the basis splits of the start and end vertex. The first case considers two vertices that are not both of the first or third category. That is, they do not both have basis split $(m_1, 0, m_2-1)$ or $(m_1-1, 0, m_2).$ 

\begin{lem}\label{nbnb1}
Let $\mathcal{P}$ be simple. Let $(x^1, s^1, y^1)$, $(x^2, s^2, y^2) \in \mathcal{P}$ be vertices, where $y^1$ is comprised of $m_2 - 1$ basic variables and $x^2$ is comprised of $m_1 - 1$ basic variables. If $ax^2 \leq ax^1 + s^1$, then the distance between $(x^1, s^1, y^1)$ and $(x^2, s^2, y^2)$ is at most $d(\bar{A}) + d(\bar{B})$. Similarly, if $x^1$ is comprised of $m_1 - 1$ basic variables, $y^2$ is comprised of $m_2 - 1$ basic variables and $ax^1 \leq ax^2 + s^2$ then the distance between $(x^1, s^1, y^1)$ and $(x^2, s^2, y^2)$ is at most $d(\bar{A}) + d(\bar{B})$.
\end{lem}

\begin{proof}
We first note $(x^2, s', y^1)$ is feasible for some $s' > 0$ because $Ax^2 = c_A$, $By^1 = c_B,$ and $ax^2 + by^1 \leq ax^1 + s^1 + by^1 = c_a + c_b.$ Moreover, by simplicity of $\mathcal{P}$ the point $(x^2, s', y^1)$ is a vertex because $x^2$ has $m_1-1$ basic variables, $s'$ is basic, and $y^1$ has $m_2 -1$ basic variables. If $(x^2, s', y^1)$ was not a vertex, this would imply the existence of a degenerate vertex. By an application of Lemma \ref{lift_lem}, the distance between $(x^1, s^1, y^1)$ and $(x^2, s', y^1)$ is at most $d(\bar{A})$.

Additionally, we note that there exists a walk of length at most $d(\bar{B})$ from $(x^2, s', y^1)$ to $(x^2, s^2, y^2)$, which again follows from an application of Lemma \ref{lift_lem}. By adjoining this walk to the previous walk (in the natural way), we have shown that there exists a walk from $(x^1, s^1, y^1)$ to $(x^2, s^2, y^2)$ with length at most $d(\bar{A}) + d(\bar{B}).$

In a polyhedron, edge walks are reversible. Thus, if we want to construct a path from $(x^1, s^1, y^1)$ to $(x^2, s^2, y^2)$ where $x^1$ is comprised of $m_1 - 1$ basic variables, $y^2$ is comprised of $m_2 - 1$ basic variables and $ax^1 \leq ax^2 + s^2$, then it is sufficient to construct a path from  $(x^2, s^2, y^2)$ to $(x^1, s^1, y^1)$. The reversed path is a walk from $(x^1, s^1, y^1)$ to $(x^2, s^2, y^2)$. By the arguments above, the length of this path is at most $d(\bar{A}) + d(\bar{B})$.
\qed
\end{proof}


Lemma \ref{nbnb1} implies the following corollary concerning the distance between pairs of vertices where the shared variable is basic for each vertex (Category 2). 

\begin{cor}\label{basicbasic}
Suppose that $\mathcal{P}$ is simple and that $(x^1, s^1, y^1), (x^2, s^2, y^2) \in \mathcal{P}$ are vertices with $s^1$ and $s^2$ basic. Then there exists an edge walk from $(x^1, s^1, y^1)$ to $(x^2, s^2, y^2)$ with length at most $d(\bar{A}) + d(\bar{B}).$ 
\end{cor}

\begin{proof}
Either $ax^1 \leq ax^2$ or $ax^2 \leq ax^1.$ Thus, one of $ax^1 \leq ax^2 + s^2$ or $ax^2 \leq ax^1 + s^1$ is satisfied. Therefore, the claim follows by Lemma \ref{nbnb1}.
\qed
\end{proof}


To prove the stated bound in Theorem \ref{PCDia}, we reduce the remaining cases to the case addressed in Corollary \ref{basicbasic}. We demonstrate that this reduction can be accomplished with at most 2 additional steps in the walk.


\begin{thm} \label{PCDia}
The combinatorial diameter of $\mathcal{P}$ is at most $d(\bar{A}) + d(\bar{B}) + 2$.
\end{thm}

\begin{proof}
We assume that $\mathcal{P}$ is simple.
Let $(x^1, s^1, y^1), (x^2, s^2, y^2) \in \mathcal{P}$ be vertices. If $s^1$ and $s^2$ are basic, then the distance between $(x^1, s^1, y^1)$ and $(x^2, s^2, y^2)$ is at most $d(\bar{A}) + d(\bar{B})$ by Corollary \ref{basicbasic}. 

For $i \in \{ 1, 2 \}$, if $s^i$ is nonbasic, then $(x^i, s^i, y^i)$ is adjacent to a vertex where $s$ is basic. Therefore, by Corollary \ref{basicbasic}, the distance between $(x^1, s^1, y^1)$ and $(x^2, s^2, y^2)$ is at most $d(\bar{A}) + d(\bar{B}) + 2.$
\qed
\end{proof}

The diameter of the parallel connection of two Hirsch satisfying polyhedra does not exceed the Hirsch bound by more than 3. First, we note that the Hirsch bound is $m_1 + m_2 - 1$, which follows because the polyhedron is simple and has $m_1 + m_2 - 1$ equality constraints. By Theorem \ref{PCDia}, the diameter is at most $m_1 + m_2 + 2$ since $d(\bar{A}) \leq m_1$ and $d(\bar{B}) \leq m_2$. The result is stated below, without proof. 


\begin{cor}[Hirsch - Parallel Connection]
Let $Q, R$ be polyhedra that both satisfy the Hirsch bound for all right-hand sides. Then, the combinatorial diameter of $\mathcal{P}$ does not exceed the Hirsch bound by more than 3. 
\end{cor}






To refine our analysis, we show that the distance between two vertices of the same category is at most $d(\bar{A}) + d(\bar{B}) + 1$. This fact, along with Lemma \ref{nbnb1}, shows that the bottleneck case in our bound is the combination of two vertices, one in Category 1 and the other in Category 3, and the inequality comparing $ax^1$ and $ax^2$ in Lemma \ref{nbnb1} is not satisfied. An improvement to this case would improve the bound to  $d(\bar{A}) + d(\bar{B}) + 1$. 

\begin{lem} \label{nocrossover}
Suppose that $(x^1, s^1, y^1)$ and $(x^2, s^2, y^2)$ are vertices of $\mathcal{P}$. If both $x^1$ and $x^2$ are comprised of $m_1 - 1$ basic variables, then the distance between $(x^1, s^1, y^1)$ and $(x^2, s^2, y^2)$ is bounded by $d(\bar{A}) + d(\bar{B}) + 1$. The same bound can be derived if both $y^1$ and $y^2$ are comprised of $m_2 - 1$ basic variables. 
\end{lem}

\begin{proof}
We assume that $\mathcal{P}$ is simple.
If both $s^1$ and $s^2$ are basic, then by Corollary \ref{basicbasic} the distance is bounded by $d(\bar{A}) + d(\bar{B})$. If only $s^1$ is basic, then $s^2$ is identically zero, and thus by Lemma \ref{nbnb1}, the distance between $(x^1, s^1, y^1)$ and $(x^2, s^2, y^2)$ is at most $d(\bar{A}) + d(\bar{B}).$

If both $s^1$ and $s^2$ are nonbasic, then an extra step is required. For the first step, increase the coordinate $s$ of $(x^1, s^1, y^1).$ As $x^1$ is comprised of $m_1 - 1$ basic variables and any feasible solution must satisfy $Ax = c_A$, it follows that $x^1$ must remain fixed. After increasing $s$, the resulting vertex, denoted $(x^1, s', y')$, must have basis split $(m_1-1, 1, m_2-1).$ 

Again by Lemma \ref{nbnb1}, the distance between $(x^1, s', y')$ and $(x^2, s^2, y^2)$ is at most $d(\bar{A}) + d(\bar{B})$. Because $(x^1, s', y')$ is adjacent to $(x^1, s^1, y^1)$, it follows that the distance between $(x^1, s^1, y^1)$ and $(x^2, s^2, y^2)$ is at most $d(\bar{A}) + d(\bar{B}) + 1.$
\qed
\end{proof}



{
\renewcommand{\arraystretch}{1.1}
\begin{table}
\small
\title{}
    \centering
        \begin{tabular}{c|c|c|c|c}
            \textbf{  Initial Basis Split } & \textbf{Final Basis Split} & \textbf{  Inequality  }  & \textbf{  Diam. Bound  } & \textbf{ Reference } \\
            \hline 
            $(m_1-1, 1, m_2-1)$ & $(m_1-1, 1, m_2-1)$ & N/A  & $d(\bar{A}) + d(\bar{B})$ & \textbf{ Corollary \ref{basicbasic} }\\ 
            \hline
            $(m_1-1, 1, m_2-1)$ & $(m_1-1, 0, m_2)$ & $ax^1 \geq ax^2$ & $d(\bar{A}) + d(\bar{B})$ & \textbf{ Lemma \ref{nbnb1} }\\ 
            \hline
            $(m_1-1, 1, m_2-1)$ & $(m_1, 0, m_2-1)$ & $ax^1 \leq ax^2$  & $d(\bar{A}) + d(\bar{B})$ & \textbf{ Lemma \ref{nbnb1} } \\ 
            \hline
            $(m_1-1, 1, m_2-1)$ & $(m_1-1, 0, m_2)$ & $ax^1 < ax^2$ & $d(\bar{A}) + d(\bar{B}) + 1$ & \textbf{ Theorem \ref{PCDia} } \\ 
            \hline
            $(m_1-1, 1, m_2-1)$ & $(m_1, 0, m_2-1)$ & $ax^1 > ax^2$  & $d(\bar{A}) + d(\bar{B}) + 1$ & \textbf{ Theorem \ref{PCDia} } \\ 
            \hline
            $(m_1, 0, m_2-1)$ & $(m_1, 0, m_2-1)$ & N/A & $d(\bar{A}) + d(\bar{B}) + 1$ & \textbf{ Lemma \ref{nocrossover} } \\
            \hline
            $(m_1-1, 0, m_2)$ & $(m_1-1, 0, m_2)$ & N/A  & $d(\bar{A}) + d(\bar{B}) + 1$ & \textbf{ Lemma \ref{nocrossover} } \\
            \hline
            
            $(m_1, 0, m_2-1)$ & $(m_1-1, 0, m_2)$ & $ax^1 \geq ax^2$  & $d(\bar{A}) + d(\bar{B})$ & \textbf{ Lemma \ref{nbnb1} } \\
            \hline
            $(m_1-1, 0, m_2)$ & $(m_1, 0, m_2-1)$ & $ax^1 \leq ax^2$  & $d(\bar{A}) + d(\bar{B})$ & \textbf{ Corollary \ref{nbnb1} } \\
            \hline
            $(m_1, 0, m_2-1)$ & $(m_1-1, 0, m_2)$ & $ax^1 < ax^2$  & $d(\bar{A}) + d(\bar{B}) + 2$ & \textbf{ Theorem \ref{PCDia} } \\
            \hline
            $(m_1-1, 0, m_2)$ & $(m_1, 0, m_2-1)$ & $ax^1 > ax^2$  & $d(\bar{A}) + d(\bar{B})+ 2$ & \textbf{ Theorem \ref{PCDia} } \\
        \end{tabular}
        \caption{Diameter bounds for the parallel-connection polyhedron by each case of basis split for initial and final vertex.}
        \label{tab:parallel}
\end{table}
}

Table \ref{tab:parallel} gives an exhaustive list of cases and the resulting diameter bounds. The cases are distinguished by categories of vertices and whether or not $ax^1 \leq ax^2$.


\section{Series Connection} \label{SeriesSec}

In this section, we establish two bounds for the diameter of a series-connection polyhedron. The first bound requires $\mathcal{S}$ to be an integral polyhedron, which is not a restriction for any polyhedron with rational input. The bound is presented in terms of the maximum and minimum values that $s$ may take. Bounds of this type are typical in the studies of lattice polytopes. For example, Deza and Pournin proved that the diameter of a lattice polytope $P$ is at most $kd - \lceil 2/3d \rceil$ where $P \subseteq [0, k]^d$ when $k > 2$ \cite{deza2018improved}. 

The second bound relaxes the requirement that $\mathcal{S}$ is integral but requires a type of boundedness for pairs of vertices where the path we want to lift violates a condition we require on $s$. We will call this type of boundedness the $s$-bounded diameter (Definition \ref{def:boundeddiam}). In turn, we establish a bound that can be expressed in the diameters and $s$-bounded diameters of $Q$ and $R$.

For the series connection, we require the walks in $Q$ and $R$ to be {\em non-revisiting}. As a service to the reader, we recall some background. This property is typically only used in the context of simple polyhedra because each step of a walk leaves exactly one facet and enters exactly one facet.

\begin{defi}
Let $P$ be a simple polyhedron. A walk in $P$ is non-revisiting if every step of the walk enters a new facet, not previously visited before.  
\end{defi}

The non-revisiting conjecture asks if there exists a non-revisiting walk between every pair of vertices of a polyhedron. If a polyhedron has this property, we say the polyhedron {\em satisfies non-revisiting}. The non-revisiting conjecture is intimately related to the Hirsch conjecture. The conjectures are, in fact, equivalent \cite{kw-67}.

\begin{prop} \label{nonrevis_hirsch}
Let $P$ be a simple polyhedron. $P$ satisfies the Hirsch bound if and only if $P$ satisfies the non-revisiting conjecture. 
\end{prop}

The proof of Proposition \ref{nonrevis_hirsch} (see \cite{kw-67}) demonstrates that if a walk is non-revisiting, then the length of the walk is bounded above by the Hirsch bound. Because the length of a non-revisiting walk is at most the Hirsch bound, we assume that the walks that  we consider also satisfy non-revisiting. For the remainder of this section, we assume that $Q$ and $R$ satisfy the non-revisiting conjecture, and hence the Hirsch bound, for all right-hand sides.

The non-revisiting property also implies that if $v$ and $w$ are vertices on the same face, then there exists a walk from $v$ to $w$ with length bounded by Hirsch that does not leave the face. We use this property in our work to construct a walk between two vertices both on the facet defined by $s = 0$. The non-revisiting property ensures that the walk does not leave the facet and still has length bounded by Hirsch. This property is derived from the fact that if a simple polyhedron has the non-revisiting property then every face of the polyhedron has the non-revisiting property. The idea generalizes further. We demonstrate that if a polyhedron (not necessarily simple) satisfies Hirsch for all right-hand sides, then any face also satisfies Hirsch for all right-hand sides. We prove a particular formulation of this observation below.


\begin{lem} \label{degeneracy_lem}
Let $A' := \begin{bmatrix} A & a \end{bmatrix}$ be a matrix such that any polyhedron with constraint matrix $A'$ satisfies the Hirsch bound for all right-hand sides. Then, a polyhedron with constraint matrix $A$ satisfies the Hirsch bound for all right-hand sides. 
\end{lem}

\begin{proof}
Let $P' := \{ x : A'x = b, x \geq 0 \}$ and $P := \{ x : Ax = b, x \geq 0 \}$ for some $b \in \R^{m}.$ Let $k$ denote the Hirsch bound for $A'$. If $P$ is empty, then $P$ satisfies the Hirsch bound vacuously. 

We suppose that $P$ is nonempty. If $P'$ is simple, then $P'$ satisfies non-revisiting. Therefore, $P$ satisfies non-revisiting, and thus Hirsch, because $P$ is a face of $P'$. Now, we assume that $P'$ is degenerate. Let $y \in \R^m$ be a perturbation of $P'$ such that $P'(y) := \{ x : A'x = b+y, x \geq 0 \}$ is simple and has the same set of feasible bases as $P'$. The polyhedron $P(y) := \{ x : Ax = b+y, x \geq 0 \}$ is a face of $P'(y)$. Thus $P(y)$ is simple and satisfies non-revisiting.

Let $v^1, v^2 \in P$ be vertices and $v_y^1, v_y^2 \in P(y)$ be resulting vertices after perturbing $v^1, v^2$ (vertices with the same basis). Because $P(y)$ satisfies non-revisiting, there exists a non-revisiting walk between $v_y^1$ and $v_y^2$ of length at most $k.$ If two vertices are adjacent in $P(y)$, then the corresponding vertices in $P$ are either adjacent or equal. Thus, the distance from $v^1$ to $v^2$ is at most $k.$ Therefore, $P$ satisfies Hirsch.
\qed
\end{proof}




Similar to our work with $\mathcal{P}$, we will refer to points of $\mathcal{S}$ as triples $(x, s, y)$ where $x \in \R^{n_1 - 1}$, $s \in \R$, $y \in \R^{n_2 - 1}$. With this formulation, we have: 
    \[
        \mathcal{S} = 
        \left\{ 
            \begin{pmatrix}
                x \\
                s \\
                y
            \end{pmatrix} \in \R^{n_1 + n_2 - 1} : 
        \begin{bmatrix}
            A & 0 & 0 \\
            a & 1 & 0 \\
            0 & 1 & b \\
            0 & 0 & B
        \end{bmatrix} \begin{pmatrix}
                x \\
                s \\
                y
            \end{pmatrix} =
            \begin{pmatrix}
                c_A \\
                c_a \\
                c_b \\
                c_B
            \end{pmatrix},
            (x, s, y) \geq 0
        \right\}. 
    \]
    Recall, we assume that $\bar{A}$ and $\bar{B}$ have full rank and $\mathcal{S}$ is simple. Additionally, we note the series connection is, in fact, more general than we have presented.
    If $M$ and $N$ are any pair of matrices then the definition of $S(\cdot, \cdot)$ can be extended to $M, N$ in the following fashion:
    \[
    M := \begin{bmatrix}
        M' & m_1 \\
        m_2 & m
    \end{bmatrix}
    , \quad
    N := \begin{bmatrix}
        n & n_2 \\
        n_1 & N'
    \end{bmatrix}
    , \quad
    S(M, N) := \begin{bmatrix}
        M' & m_1 & 0 \\
        m_2 & m & 0 \\
        0 & n & n_2 \\
        0 & n_1 & N'
    \end{bmatrix}.
\]
 If the final column of $M$ and the first column of $N$ are nonzero, then through row operations $S(M, N)$ can be reduced to $S(\bar{A}, \bar{B})$ for some $\bar{A}, \bar{B}$. As the graph of $\mathcal{S}$ is not affected by row operations, all results concerning the diameter still hold. However, if the last column of $M$ or first column of $N$ is the zero vector, then the resulting series connection reduces to a Cartesian product. The diameter bounds we establish still hold for this case.


\subsection{Categorization of Vertices}\label{sec:classseries}

We first establish an understanding of the graph of $\mathcal{S}$; see Definitions \ref{def:S} and \ref{def:SQR}. Similar to the parallel-connection polyhedron, we categorize the vertices through an observation on the number of basic variables in the $x$- and $y$-blocks. The vertices of $\mathcal{S}$ are in one-to-one correspondence with the basic feasible solutions of $\mathcal{S}$ because $\mathcal{S}$ is simple. Any basic feasible solution has $m_1 + m_2$ (nonzero) basic variables because the simple $\mathcal{S}$ has $m_1 + m_2$ equality constraints. By simplicity of $Q$ and $R$, the equations $Ax = c_A$ and $By = c_B$ each require $m_1-1$ and $m_2-1$ (nonzero) basic variables, respectively, to be satisfied. Therefore, a vertex $(x, s, y) \in \mathcal{S}$ has at least $m_1-1$ basic variables in the $x$-block and $m_2-1$ basic variables in the $y$-block. Thus, we have the following categorization (recall the notation of $Q(t)$ and $R(t)$ introduced in Subsection \ref{PCDiaBd}):

\begin{prop} \label{seriesVerClass}
If $\mathcal{S}$ is simple and has an equality constraint matrix with linearly-independent rows, then the vertices of $\mathcal{S}$ can be partitioned into the following three categories:
\begin{enumerate}
    \item The vector $(x, 0)$ is a vertex of $Q$, the variable $s$ is nonbasic, and the vector $(0, y)$ is a vertex of $R.$
    \item The variable $s$ is basic, for $t: = c_a - s$, the vector $(x, 0)$ is a vertex of $Q(t)$, and the vector $(s, y)$ is a vertex of $R.$
    \item The variable $s$ is basic, the vector $(x, s)$ is a vertex of $Q$, and for $t := c_b - s$, the vector $(0, y)$ is a vertex of $R(t).$
\end{enumerate}
\end{prop}




The above categorization of the vertices captures the possible basis splits in the respective blocks and relates the coordinate vectors to vertices of the original polyhedra. As we did in Section \ref{PCSec}, we will establish diameter bounds by breaking the problem into subcases based on these categories.

\subsection{Lattice-Type Diameter Bound} \label{lattice_bound}

In this section, we establish a diameter bound for the series-connection polyhedron that relies on the difference between the maximum and minimum values of the shared variable $s.$ Such bounds typically arise in the studies of lattice polytopes or general polyhedra in combinatorial optimization; see, e.g., \cite{b-13,del2016diameter,deza2018improved}. The state-of-the-art bounds for the diameter of a lattice polytope in $[0, k]^d$ are $kd$ for $k=1$ \cite{kleinschmidt1992diameter}, $\lfloor (kd - (1/2)d) \rfloor$ when $k = 2$ \cite{del2016diameter}, and $kd -  \lceil(2/3)d \rceil$ when $k \geq 3$ \cite{deza2018improved}.

In what follows, we will make reference to the maximum and minimum values that $s$ can take on in $Q, R$, and $\mathcal{S}.$ We set $s_{\min}^Q := \min \{ s : (x, s) \in Q \}$ and $s_{\max}^Q := \max \{ s : (x, s) \in Q \}.$ We define $s_{\min}^R$ and $s_{\max}^R$ analogously. 
We show that the diameter of $\mathcal{S}$ is bounded above by $m_1 + m_2 + s_{\text{diff}}$ where:
\[
s_{\text{diff}} := \max\{s : (x, s, y) \in \mathcal{S} \} - \min\{s : (x, s, y) \in \mathcal{S}\}.
\]
We note that the range of values $s$ that may take decreases following the series-connection operation.

\begin{lem}\label{lem:srange}
Let $(x, s, y) \in \mathcal{S}$. Then $
    \max \{ s_{\min}^Q, s_{\min}^R \}
    \leq s \leq 
    \min \{ s_{\max}^Q, s_{\max}^R  \}. 
$
Moreover, these bounds are tight.
\end{lem}

\begin{proof}
Let $(x, s, y) \in \mathcal{S}.$ Then, $(x, s) \in Q$ and $(s, y) \in R$. Therefore,
\[
    s_{\min}^Q \leq s \leq s_{\max}^Q \text{ and } s_{\min}^R \leq s \leq s_{\max}^R.
\]
Thus, we conclude 
\[
    \max \{ s_{\min}^Q, s_{\min}^R \}
    \leq s \leq 
    \min \{ s_{\max}^Q, s_{\max}^R  \}.
\]
For tightness, we assume, without loss of generality, that $s_{\max}^Q \leq s_{\max}^R.$ By definition of $s_{\max}^Q$, there exists an $x$ such that $(x, s_{\max}^Q) \in Q.$ Additionally, there exist $y^1, y^2$ such that $(s_{\min}^R, y^1)$, $(s_{\max}^R, y^2) \in R.$ As $s_{\min}^R \leq s_{\max}^Q \leq s_{\max}^R$, we can find $\lambda \in [0, \, 1]$ such that $s_{\max}^Q =\lambda s_{\min}^R + (1 -\lambda)s_{\max}^R.$ Setting $y = \lambda y^1 + (1 - \lambda)y^2$, it follows by convexity that $(s_{\max}^Q, y) \in R.$ Therefore, $(x, s_{\max}^Q, y) \in \mathcal{S}.$ A similar argument shows that the bound $\max \{ s_{\min}^Q, s_{\min}^R \}$ is tight. 
\qed
\end{proof}

Our technique for Theorem \ref{SQR_diam} relies on concatenating three walks to form a walk between a pair of vertices. Two of the portions connect the start and end vertices, respectively, to vertices where $s$ is maximal or minimal. We show that such a walk exists and has length at most $s_{\rm diff}$ for any integral polyhedron $S$.

\begin{lem} \label{walklem}
Suppose that $S$ is an integral polyhedron with variable $s_{\min} \leq s \leq s_{\max}.$ Let $(x, s, y) \in S$ be a vertex. Then $(x, s, y)$ is at most a distance $s_{\rm{diff}}$ from a vertex where $s$ is minimal and from a vertex where $s$ is maximal.
\end{lem}

\begin{proof}
Let $(x', s', y')$ be a vertex. Consider the linear programs $\min \{ s : (x, s, y) \in S \}$ and $\max \{ s : (x, s, y) \in P \}.$ 
The simplex method initialized at $(x',s',y')$ can be used to construct a walk to an optimal vertex. The value of $s$ must improve by at least one for each (non-degenerate) step because $S$ is integral and the simplex method strictly improves the objective function. Therefore, $\min \{ s : (x, s, y) \in S \}$ is solved in at most $s' - s_{\min}$ steps, and $\max \{ s : (x, s, y) \in S \}$ is solved in at most $s_{\max} - s'$ steps. In both cases, the number of steps is bounded by $s_{\rm{diff}}.$  
\qed
\end{proof} 

Lemma \ref{walklem} is the key ingredient to devise a bound of $\mathcal{S}$ in terms of $s_{\rm diff}.$ 

\begin{thm} \label{SQR_diam}
Let $Q$ and $R$ be integral polyhedra that satisfy the non-revisiting conjecture. Let $\mathcal{S}$ be the series-connection polyhedron for $Q$ and $R$ and $s_{\min} := \min \{s : (x, s, y) \in \mathcal{S}\}$.
If $(x^1, s^1, y^1)$, $(x^2, s^2, y^2) \in \mathcal{S}$ are vertices, then the distance between $(x^1, s^1, y^1)$ and $(x^2, s^2, y^2)$ is at most $m_1 + m_2 + 2s_{\rm{diff}}.$ If $\sup \{s : (x, s, y) \in \mathcal{S}\} < \infty$, then the diameter of $\mathcal{S}$ is bounded by $m_1 + m_2 + s_{\text{diff}}$.
\end{thm}

\begin{proof}
We assume that $\mathcal{S}$ is simple.
First, we show that if $(x^1, s^1, y^1), (x^2, s^2, y^2) \in \mathcal{S}$ are vertices where $s^1 = s^2$ and $s^1, s^2$ are both maximal or both minimal then the distance between $(x^1, s^1, y^1)$ and $(x^2, s^2, y^2)$ is at most $m_1 + m_2.$
Let $(x^1, s^1, y^1), (x^2, s^2, y^2)$ be such vertices. We assume that $(x^1, s^1, y^1) \neq (x^2, s^2, y^2)$, otherwise the distance between $(x^1, s^1, y^1)$ and $(x^2, s^2, y^2)$ is zero. It follows that
\[
    (x^1, y^1), (x^2, y^2) \in \mathcal{S}' := \left\{ (x, y) \in \R^{n_1 + n_2 - 2} : 
    \begin{bmatrix}
        A & 0 \\
        a & 0 \\
        0 & b \\
        0 & B
    \end{bmatrix}
    \begin{bmatrix}
    x \\
    y 
    \end{bmatrix}
    =
    \begin{bmatrix}
    c_A \\
    c_a - s^1\\
    c_b - s^1\\
    c_B
    \end{bmatrix},
    (x, y) \geq 0
    \right\}
\]
and $(x^1, y^1), (x^2, y^2)$ are vertices of $\mathcal{S}'$. Further, $\mathcal{S}'$ is a face of $\mathcal{S}$, because $s$ is at an extreme value. Therefore, any path in $\mathcal{S}'$ can be realized as a path in $\mathcal{S}.$ We note the constraint matrix of $\mathcal{S}'$ is the $1$-sum of $\begin{bmatrix} A \\ a \end{bmatrix}$ and $\begin{bmatrix} b \\ B \end{bmatrix}.$ Thus, the diameter is at most $d(\begin{bmatrix} A \\ a \end{bmatrix}) + d(\begin{bmatrix} b \\ B \end{bmatrix}).$ When nonrevisiting is satisfied, the bound reduces to $m_1 + m_2$. 

Let $(x^1, s^1, y^1),$ $(x^2, s^2, y^2) \in \mathcal{S}$ be vertices. Now, we establish that the distance between $(x^1, s^1, y^1)$ and $(x^2, s^2, y^2)$ is at most $m_1 + m_2 + 2s_{\rm diff}$. By integrality of $\mathcal{S}$, any step which decreases the value of $s$ will decrease the value of $s$ by at least $1$. Therefore, the vertex $(x^1, s^1, y^1)$ is at most $s^1$ steps from a vertex where $s$ is minimal. Similarly, $(x^2, s^2, y^2)$ is at most $s^2$ steps from a vertex where $s$ is minimal. Thus, the distance between $(x^1, s^1, y^1), (x^2, s^2, y^2)$ is at most $m_1 + m_2 + s^1 + s^2 \leq m_1 + m_2 + 2s_{\rm{diff}}.$ 

We now suppose that $\sup \{s : (x, s, y) \in \mathcal{S}\} < \infty$.
By Lemma \ref{walklem}, it follows that $(x^1, s^1, y^1)$ is at most a distance $s^1 - s_{\min}$ from a vertex where $s$ is minimal and at most a distance $s_{\max} - s^1$ from a vertex where $s$ is maximal.
For $(x^2, s^2, y^2)$ there are similar bounds ($s^2 - s_{\min}$ and $s_{\max} - s^2$). It can be shown that $(x^1, s^1, y^1)$ and $(x^2, s^2, y^2)$ are at most a combined $s_{\rm{diff}}$ steps from vertices where $s$ is either maximal or minimal. Thus, the distance between $(x^1, s^1, y^1), (x^2, s^2, y^2)$ is at most $m_1 + m_2 + s_{\rm{diff}}$. 
\qed
\end{proof}


We conclude with a brief remark that Theorem \ref{SQR_diam} implies that if $Q$ and $R$ satisfy the Hirsch bound for all right-hand sides, then the diameter of $\mathcal{S}$ does not exceed Hirsch by more than $s_{\rm{diff}}$. This result follows because the Hirsch bound for $\mathcal{S}$ is $m_1 + m_2$ when $\mathcal{S}$ is simple.


\subsection{A Quadratic Bound} \label{sec:series_bound}

In the previous subsection, the diameter of $\mathcal{S}$ was bounded using the quantity $s_{\rm{diff}}.$ However, a scaling of a polyhedron by a factor of $\lambda>0$ would also scale $s_{\rm{diff}}$ with $\lambda$, even though the diameter does not change. In this subsection, we establish a bound that does not rely on $s_{\rm{diff}}.$  

Our technique ``corrects'' a step that failed to lift. Suppose that $(x^0, s^0, y^0), (x', s', y') \in \mathcal{S}$ are vertices with basis split $(m_1-1, 1, m_2).$ By Proposition \ref{seriesVerClass}, $(x^0, s^0), (x', s')$ are vertices of $Q$. Thus, there exists a walk between $(x^0, s^0), (x^1, s^1), \ldots, (x^k, s^k) = (x', s')$ in $Q$. If the entire walk lifts successfully to $\mathcal{S}$, then constructing a diameter bound that is linear in $d(\bar{A})$ and $d(\bar{B})$ is straightforward with our previous techniques. If a step $(x^i, s^i)$ to $(x^{i+1}, s^{i+1})$ does not lift successfully, meaning $(x^{i+1}, s^{i+1}, y^0)$ is infeasible, then an additional $d_b(R)$ (Definition \ref{def:boundeddiam}) steps may be needed to reach a vertex $(x^{i+1}, s^{i+1}, y)$ for some $y$. After this ``correction'', we proceed by lifting the step $(x^{i+1}, s^{i+1})$ to $(x^{i+2}, s^{i+2})$ and using an additional $d_b(R)$ steps to correct, if needed. Lemmas \ref{correction_lemma} and \ref{crossover_lemma} make these arguments rigorous. Theorem \ref{series_corr_diam} is our diameter bound. 


\begin{defi}\label{def:boundeddiam}
Let $P$ be a polyhedron and let $(x, s^0) \in P$ be a vertex. The $s_{\max}$-bounded distance of $(x, s^0)$ is the minimal number of steps to walk from $(x, s^0)$ to a vertex with $s$ maximal where, for every vertex $(x^i, s^i)$ in the walk, $s^i > s^0.$ The $s_{\min}$-bounded distance of $v$ is defined similarly for a vertex where $s$ is minimal. The $s$-bounded distance of $v, d_b(v)$ is the maximum of the two distances. The $s$-bounded diameter of $P$ is $d_b(P) := \max \{d_b(v) : v \in P, v \text{ is a vertex} \}.$
\end{defi}

In what follows, we demonstrate that the diameter of $\mathcal{S}$ is bounded above by a function that is quadratic in the diameter of one of the original polyhedra and the $s$-bounded diameter of the other. 
The $s$-bounded diameter represents the weakest requirement for the construction of the bound. A more restrictive formulation is the well-studied \textit{monotone diameter}.

The $c$-monotone diameter, for a linear functional $c$, is the maximum length of a shortest path between an arbitrary vertex and a vertex that is optimal with respect to $c$ with the restriction that the value $c^Tx$ improves with each step on the path. This diameter is therefore intimately tied to the simplex method. The monotone diameter is the maximum over all choices of $c$. The $s$-bounded diameter of $\mathcal{S}$ is bounded above by the $c$-monotone diameter of $\mathcal{S}$ which is bounded above by the monotone diameter of $\mathcal{S}$.



For brevity, we overload the notation $d_b(\cdot)$ by defining $d_b(M)$ for a matrix $M$. Here, $d_b(M)$ represents the greatest $s$-bounded diameter of a polyhedron with constraint matrix $M$. Lemma \ref{correction_lemma} shows that the $s$-bounded diameter can be used to bound the distance between two particular vertices of $\mathcal{S}.$

\begin{lem}\label{correction_lemma}
If $(x^0, s^0), (x', s') \in Q$ are adjacent vertices such that there exist vertices $(x^0, s^0, y^0)$, $(x', s', y') \in \mathcal{S}$, then there exists a $y$ such that $(x', s', y) \in \mathcal{S}$ is a vertex, and the distance between $(x^0, s^0, y^0) $ and $(x', s', y)$ is at most $d_b(R) + 1.$
\end{lem}

\begin{proof}
Throughout this proof, we assume that $\mathcal{S}$ is simple.
First, we suppose that $s^0 = s' = 0.$ It is easy to verify that $(x', s', y^0) \in \mathcal{S}$ is a vertex and $(x^0, s^0, y^0)$ and $(x', s', y^0)$ are adjacent.

Now, we assume that $s^0 \geq s'$ and $s^0 > 0.$ If $(x', s', \bar{y}^0)$ is feasible for some $\bar{y}^0$ with the same zero and nonzero components as $y^0$, then by simplicity $(x', s', \bar{y}^0)$ is a vertex and $(x^0, s^0, y^0)$ and $(x', s', \bar{y}^0)$ are adjacent vertices. If no such $\bar{y}^0$ exists, then lifting the step $(x^0, s^0)$ to $(x', s')$ yields a change in basis split. We call the resulting vertex $(x^1, s^1, y^1)$, which has basis split $(m_1, 1, m_2-1)$. By Proposition \ref{seriesVerClass}, $(s^1, y^1) \in R$. 

The point $(x^1, s^1)$ lies on the edge of $Q$ defined by the convex hull of $(x^0, s^0)$ and $(x', s')$ because it arose from the step from $(x^0, s^0)$ to $(x', s')$. For any adjacent vertex $(s^2, y^2)$ of $(s^1, y^1)$, a step $(s^1, y^1) \rightarrow (s^2, y^2)$ lifts successfully at $(x^1, s^1, y^1)$ if and only if there exists $\bar{x}^1$, with the same zero and nonzero components as $x^1$, such that $(\bar{x}^1, s^2, y^2) \in \mathcal{S}.$ This is possible if and only if $(\bar{x}^1, s^2) \in \operatorname{conv}\{(x^0, s^0), (x', s') \}$; otherwise $(\bar{x}^1, s^2) \notin Q$. Therefore, $s^2 \in [s', s^0]$. Thus, we have the following criterion: a step $(s^1, y^1)$ to $(s^2, y^2)$ lifts at $(x^1, s^1, y^1)$ if and only if $s^2 \in [s', \, s].$ Otherwise, if $s^2 < s'$ the step is blocked at $(x', s', y)$ for some $(0, y) \in R$ and if $s^2 > s$ the step is blocked at $(x^0, s^0, y)$ for some $(0, y) \in R.$

Now, we construct the path. Let $(s^*,y^*)$ be a vertex representative from $\operatorname{argmin} \{ s : (s, y) \in R \}.$ We let $(s^1, y^1), (s^2, y^2), \ldots, (s^k, y^k) = (s^*, y^*)$ denote the shortest path from $(s^1, y^1)$ to $(s^*, y^*)$ satisfying $s^i \leq s^1$ for all $i.$ Existence of this path follows because such a path can be constructed via the simplex method with an improving pivot rule. By Definition \ref{def:boundeddiam}, $k \leq d_b(R)$. Let $j$ be the smallest index such that $s^j > s'$ and $s^{j+1} \leq s'$, which exists as $s^* \leq s' < s^1.$ Now, the sub-path $(s^1, y^1) \rightarrow (s^j, y^j)$ lifts successfully by the aforementioned criterion. However, by the same criterion, the step $(s^j, y^j) \rightarrow (s^{j+1}, y^{j+1})$ is blocked and the step terminates at $(x', s', y)$ for some $y \in \operatorname{conv}\{ y^j, y^{j+1} \}$. Thus, this walk is at most $1 + j \leq 1 + d_b(R)$ steps.
\qed
\end{proof}

An analogous statement holds for adjacent vertices $(s, y), (s', y') \in R$ with bound $d_b(Q) + 1.$ Lemma \ref{correction_lemma} requires the existence of a $y'$ such that $(x', s', y')$ is feasible. To show existence of such a $y'$, it will be helpful to work with a particular basis split. Lemma \ref{crossover_lemma} establishes that every vertex is at most a distance $\max\{ d(\bar{A}), d(\bar{B}) \}$ from a vertex with the `opposite' basis split. 

\begin{lem} \label{crossover_lemma}
Let $\mathcal{S}$ be simple and $(x, s, y) \in \mathcal{S}$ be a vertex with basis split $(m_1, 1, m_2-1).$ If there exists a vertex $(x', s', y') \in \mathcal{S}$ with basis split $(m_1-1, 1, m_2)$ and $s_{\max}^Q < s_{\max}^R$ or $s_{\min}^Q > s_{\min}^R$  then there exists a vertex $(x^c, s^c, y^c) \in \mathcal{S}$ with basis split $(m_1-1, 1, m_2)$ or $(m_1, 0, m_2)$ at most  distance $d(\bar{B})$ from $(x, s, y).$ 
\end{lem}

\begin{proof}
First, we assume that $s_{\max}^Q < s_{\max}^R.$ This implies that there exists $(s', y') \in R$ be a vertex with $s' > s_{\max}^Q$. By assumption and Proposition \ref{seriesVerClass}, $(s, y)$ and $(s', y')$ are vertices of $R.$ We consider a shortest walk from $(s, y)$ to $(s', y')$ in $R$, which we enumerate as $(s, y) = (s^0, y^0), (s^1, y^1),\ldots, (s^k, y^k)=(s', y').$ As this is a shortest walk, we have $k \leq d(\bar{B})$.
When lifting step $i$, a coordinate of $(s^i, y^i)$ enters the basis and a coordinate of $(x^{i-1}, s^{i-1}, y^{i-1})$ leaves the basis.
If a coordinate of $(x^{i-1}, s^{i-1})$ leaves the basis, then the vertex has basis split $(m_1-1, 1, m_2)$ or $(m_1, 0, m_2).$

Assume a coordinate of $(x, s)$ never left the basis while lifting each step. It follows that the walk terminates at $(\bar{x}, s', y')$ with basis split $(m_1-1, 1, m_2)$, where $\bar{x}$ is a vector with the same zero and nonzero components of $x$. By definition of the series-connection polyhedron, we have $(\bar{x}, s') \in Q$, which contradicts $s' > s_{\max}^Q.$ Thus, at some step of the walk a coordinate of $(x, s)$ must leave the basis. Therefore, $(x, s, y)$ is at most $d(\bar{B})$ steps from a vertex with basis split $(m_1-1, 1, m_2)$  or $(m_1, 0, m_2).$

The case $s_{\min}^Q > s_{\min}^R$ is similar; we find a walk from $(s, y)$ to a vertex of form $(s_{\min}^R, y_{\min}).$ 
\qed
\end{proof}

Lemmas \ref{correction_lemma} and \ref{crossover_lemma}, are crucial tools for the proof of the diameter bound. We will consider the case where $s_{\min}^R \leq s_{\min}^Q$, although, by interchanging the roles of $Q$ and $R$, we can derive a bound when $s_{\min}^Q < s_{\min}^R$, too. 
When either $s_{\max}^R < s_{\min}^Q$ or  $s_{\max}^Q < s_{\min}^R$, $\mathcal{S}$ is empty (see Lemma \ref{lem:srange}). Then, $\mathcal{S}$ satisfies Hirsch vacuously, so we do not examine this situation further. 

Our proof is split into two cases. The first case assumes $s \leq s_{\max}^R$ for each vertex in the walk we are lifting from $Q$. In this case, each intermediate vertex $(x^i, s^i)$ can be lifted to a vertex $(x^i, s^i, y^i) \in \mathcal{S}.$ A bound can be derived from an application of Lemma \ref{correction_lemma}. The second case assumes $s > s_{\max}^R$ for some step. In this case, we adjust the arguments used. A different bound is found depending on the case; therefore, the bound in Theorem \ref{series_corr_diam} is presented as a maximum of two terms.

\begin{thm} \label{series_corr_diam} 
The diameter of the series-connection polyhedron, $\mathcal{S}$, where $s_{\min}^R \leq s_{\min}^Q$ is
at  most the maximum of 
\[
    \{d(\bar{A})(d_b(\bar{B})+1) + d(B) + 2d(\bar{B}), \,
    (d(\bar{A})-1)(d_b(\bar{B})+1) + d(\bar{A}) + 3d(\bar{B}) + 2 \}.
\]
\end{thm}

\begin{proof}
We assume that $\mathcal{S}$ is simple.
Let $(x, s, y), (x', s', y') \in \mathcal{S}$ be vertices. First, we consider the case where $y$ and $y'$ consist of $m_2$ basic variables. By Proposition \ref{seriesVerClass}, both $(x, s), (x', s') \in Q$ are vertices, and thus, there exists a walk $(x, s) = (x^0, s^0), (x^1, s^1),\ldots ,(x^k, s^k) = (x', s')$ with $k \leq d(\bar{A}).$ 

Next, we assume that for all $i$, we have $s^i \leq s_{\max}^R.$ By Lemma \ref{correction_lemma}, the vertex $(x^0, s^0, y)$ is at most a distance $d_b(\bar{B}) + 1$ from a vertex $(x^1, s^1, y^1)$ for some $(0, y^1) \in R.$ By induction, $(x^i, s^i, y^i)$ is a vertex at most a distance $d_b(\bar{B}) + 1$ from a vertex $(x^{i+1}, s^{i+1}, y^{i+1}).$ Thus, $(x, s, y)$ is at most a distance $k(d_b(\bar{B})+1)$ from a vertex $(x', s', y^k).$ We note $(0, y^k), (0, y') \in R(c_b - s^k)$ are vertices. Thus, there exists a path of length at most $d(B)$ from $(0, y^k)$ to $(0, y')$ where $s$ remains fixed at $0$, by an application of Lemma \ref{degeneracy_lem} to the facet defined by $s=0$. Thus, the distance from $(x, s, y)$ to $(x', s', y')$ is at most $d(\bar{A})(d_b(\bar{B})+1) + d(B).$

Now, we suppose that there exists a vertex, $(x^{i+1}, s^{i+1})$, in the walk with $s^{i+1} > s_{\max}^R$. We further assume $i+1$ is minimal. This implies $s_{\max}^R < s_{\max}^Q$, and thus $s' \leq s_{\max}^R$ by Lemma \ref{lem:srange}. The distance from $(x, s, y)$ to $(x^i, s^i, y^i)$ for some $y^i$ is at most $i(d_b(\bar{B})+1)$ by the previous case and the minimality of $i+1$. As $s^{i+1} > s_{\max}^R$, the step $(x^i, s^i) \rightarrow (x^{i+1}, s^{i+1})$ does not lift successfully. Instead, the resulting vertex is $((x^{i+1})', s_{\max}^R, y^{i+1})$ for some $(x^{i+1})' \in \operatorname{conv}\{x^i, x^{i+1}\}.$ Let $j$ be the largest index such that $s^j > s_{\max}^R.$ The vertex $(x', s', y')$ is at most at distance $(k-j-1)(d_b(\bar{B})+1)+1$ from a vertex of form $((x^j)', s_{\max}^R, y^j)$, for some $(x^j)' \in \operatorname{conv}\{x^j, x^{j+1} \}$, again by the previous case and the maximality of $j$.

The proof of Theorem \ref{SQR_diam} from Subsection \ref{lattice_bound} implies that the distance between $((x^{i+1})', s_{\max}^R, y^{i+1})$ and $((x^j)', s_{\max}^R, y^j)$ is at most $d(\bar{A}) + d(\bar{B}).$ Therefore, the distance between $(x, s, y)$ and $(x', s', y')$ is at most $i(d_b(\bar{B})+1) + 1 + d(\bar{A}) + d(\bar{B}) + (k-j-1)(d_b(\bar{B})+1) + 1.$ In the worst case, $i+1 = j$ and $k = d(\bar{A})$. Thus, the bound is at most $(d(\bar{A})-1)(d_b(\bar{B})+1) + d(\bar{A}) + d(\bar{B}) + 2.$ 

If $(x, s, y)$ has basis split $(m_1, 1, m_2-1)$, then $(x, s, y)$ is at most $d(\bar{B})$ steps from a vertex in another category, by Lemma \ref{crossover_lemma}. If $(x', s', y')$ has basis split $(m_1, 1, m_2-1)$, the same is true. Thus, $2d(\bar{B})$ extra steps may be required.
\qed
\end{proof}

When $s_{\min}^Q \leq s_{\min}^R$, a bound can be derived by interchanging the roles of $Q$ and $R$. In this case the diameter is bounded by the maximum of $d(\bar{B})(d_b(\bar{A})+1) + d(A) + 2d(\bar{A})$ and $(d(\bar{B})-1)(d_b(\bar{A})+1) + d(\bar{A}) + 3d(\bar{B}) + 2$. 
    
If both $Q$ and $R$ satisfy Hirsch, then the bounds can be simplified further. 

\begin{cor}\label{cor:finalseriesbound}
Suppose that $Q$ and $R$ are polyhedra that satisfy Hirsch and $s_{\min}^R \leq s_{\min}^Q$. Then the diameter of the series-connection polyhedron, $\mathcal{S}$, is at most the maximum of
\begin{align*}
    \{ d(\bar{A})(d_b(\bar{B})+1) + 3d(\bar{B}), \,
    (d(\bar{A})-1)(d_b(\bar{B})+1) + d(\bar{A}) + 3d(\bar{B}) + 2 \}.
\end{align*}
\end{cor}

The result follows from the fact that Hirsch implies that $d(A) \leq d(\bar{A})$ and $d(B) \leq d(\bar{B})$, which follows from its equivalence with the non-revisiting conjecture.

\section{Conclusion and Outlook}\label{sec:conclusion}

In this work, we established diameter bounds for the parallel and series connection of polyhedra. These constructions are based on the corresponding classical matroid operations. For the parallel-connection polyhedron, we were able to devise a bound that adds a constant $2$ to the diameters of the original polyhedra (and would exceed the claimed Hirsch bound by $3$). For the series-connection polyhedron, we obtained a bound that adds the range of the shared variable to the original diameters.  The latter can be translated into a quadratic term in the Hirsch bound, provided that distances are realized by so-called $s$-monotone walks; this is a much weaker assumption than working with general monotone diameters. While the series-connection bounds require that the underlying polyhedra satisfy the Hirsch conjecture, and while the derived bounds are weaker, they hold for a broader class of polyhedra. 

There are a few natural directions for next steps. First, validity of the Hirsch conjecture for the connection of two Hirsch-satisfying polyhedra remains open. In particular, in all but a few special cases, the constant added in the parallel connection is a strict overestimation. For the series-connection polyhedron, there is even more room for improvement towards a linear bound or the actual Hirsch bound. Corollary \ref{cor:finalseriesbound} also raises a more general question about the relation of the $s$-bounded diameter and the traditional combinatorial diameter: if the $s$-bounded diameter is polynomial in the traditional diameter, then does the series connection of two polyhedra that satisfy the polynomial Hirsch conjecture also satisfy the polynomial Hirsch conjecture? By design, our approach is limited to transferring walks in the original polyhedra to the connected polyhedron. For possible improvements, we expect that one has to leave this setting; i.e., one has to construct walks that are known to be short, but do not correspond to walks in the original polyhedra.

Second, parallel and series connections are, in particular, meaningful due to their intimate relation to the $2$-sum \cite{oxley-06}. While one can obtain a $2$-sum through deletion or contraction on the matrices, such operations generally do not allow for the immediate transfer of diameter bounds: it is well-known that faces of polyhedra that violate the Hirsch conjecture may themselves satisfy the bound and, conversely, a polyhedron may satisfy the Hirsch conjecture, but some of its faces may not. However, our approach to the parallel connection itself actually comes close to being a bound for one of the possible phrasings of $2$-sums as in \cite{s-98}; see the introduction of Section \ref{PCSec}. We essentially proved a bound that holds if one of the two matrices has a special form -- the final step is to generalize the special form to this matrix. 

And third, general $k$-sums for $k>2$ and the associated parallel and series connections on (complete) $k$-edge subgraphs 
are of similar interest to $2$-sums. In a mathematical programming sense, they correspond to a more general linking of two subsystems; in the construction of totally-unimodular matrices, both $2$- and $3$-sums are fundamental operations. It would be promising to extend our approach  -- the transfer of walks from the original polyhedra -- to $3$-sums and general $k$-sums. We expect that polynomial bounds of at most power $k$ are possible for both types of connections. 


\subsection*{Acknowledgments}

We would like to thank Antoine Deza for insightful discussions about the diameters of lattice polytopes.


This work was supported by Air Force Office of Scientific Research grant FA9550-21-1-0233 and NSF grant 2006183, Algorithmic Foundations, Division of Computing and Communication Foundations. 

\bibliography{literature}
\bibliographystyle{plain}

\end{document}